\documentclass[10pt]{article}%
\usepackage{amsmath}
\usepackage{amsfonts}
\usepackage{mathrsfs}
\usepackage{amssymb, color}
\usepackage[linkcolor=black,anchorcolor=black,citecolor=black]{hyperref}
\usepackage{graphicx}
\numberwithin{equation}{section}
\usepackage[body={15.5cm,21cm}, top=3cm]{geometry}%
\providecommand{\U}[1]{\protect\rule{.1in}{.1in}}
\providecommand{\U}[1]{\protect \rule{.1in}{.1in}}
\newtheorem{theorem}{Theorem}[section]

\newtheorem{corollary}[theorem]{Corollary}

\newtheorem{definition}[theorem]{Definition}

\newtheorem{lemma}[theorem]{Lemma}

\newtheorem{proposition}[theorem]{Proposition}
\newtheorem{remark}[theorem]{Remark}

\newtheorem{assumption}[theorem]{Assumption}
\newenvironment{proof}[1][Proof]{\noindent \textbf{#1.} }{\  \rule{0.5em}{0.5em}}
\DeclareMathOperator*{\esssup}{ess\,sup}
\DeclareMathOperator*{\essinf}{ess\,inf}
\begin{document}
	\title{Optimal Multiple Stopping Problems under $g$-expectation}
	\author{ Hanwu Li\thanks{Center for Mathematical Economics, Bielefeld University,
			hanwu.li@uni-bielefeld.de.  This work was supported by the German Research Foundation (DFG) via CRC1283.}}
			\date{}
		\maketitle
		
		\begin{abstract}
			In this paper, we study the optimal multiple stopping problem under Knightian uncertainty both under discrete-time case and continuous-time case. The Knightian uncertainty is modeled by a single real-valued function $g$, which is the generator of a kind of backward stochastic differential equations (BSDEs). We show that the value function of the multiple stopping problem coincides with the one corresponding to a new reward sequence or process. For the discrete-time case, this problem can be solved by an induction method which is a straightforward generalization of the single stopping theory. For the continuous-time case, we furthermore need to establish the continuity of the new reward family. This result can be applied to the pricing problem for swing options in financial markets, which give the holder of this contract at least two times rights to exercise it. 
		\end{abstract}
	
	\textbf{Key words}: optimal multiple stopping, Knightian uncertainty, $g$-expectation 
	
	\textbf{MSC-classification}: 60G40
	
	\section{Introduction}
	
	The optimal stopping problem is one of the fundamental problems in probability theory with wide applications, such as pricing for American options, the entry decision of a firm into a new market and the Wald test. The main objective of this problem are two manifolds. First, we would like to compute the value function as explicitly as possible. Second, we would try to find the optimal stopping time at which the value is attained. In order to solve this problem, there are two approaches: the backward induction and the Snell envelope, where both approaches can be applied to the finite discrete time case and for the other cases, only the Snell envelope is effective. For more details, we may refer to the paper \cite{E}, \cite{PS} and the references therein.
	
	In practice, there are some financial derivatives such that the holder of this contract can exercise several times before the maturity date, for example, the swing contract in energy market. To avoid mathematical trivialities, we assume that there exists some constant $\delta>0$ such that the difference of any two successive exercise times should be no less than it, which means that one can not exercise more than one right at the same time. There are tremendous papers working the multiple optimal stopping problem. To name a few, \cite{H} studies this problem in the discrete time case and then \cite{CT} investigates the continuous time case. Both of them obtained the existence of the optimal stopping times by considering a family of auxiliary optimal stopping problems with single exercise right. \cite{MH} exhibits a dual representation for the marginal value of the multiple stopping problem, which gives a natrual upper bound and therefore, can be applied to the numerical approximation.
	
	It is worth pointing out that, all the works listed above are based on the assumption that there is one and only one probability measure under consideration. It means that the investor have fully confident in the future state of the payoff in financial market. However, it is difficult to assess the probability in the real world, in other words, the investor faces Knightian uncertainty or ambiguity. Compared with the classical case, the investor will take a set of beliefs to measure the future payoff, which indicates that the evaluation is nonlinear. Therefore, in order to introduce the optimal stopping problem under Knightian uncertainty, one need to first investigate the set of beliefs or the associated nonlinear expectations. For the multiple priors approach, we may refer to the paper \cite{ETZ}, \cite{NZ}, \cite{R}, with the usual assumption of stable under pasting or rectangularity. The second group relies on the nonlinear expectation approach, such as \cite{BY1}, \cite{BY2}, \cite{CR}, \cite{LR}, where the nonlinear operator satisfies time consistency, local property and translation invariance. It is important to note that the multiple priors in \cite{ETZ} and \cite{NZ} are non-dominated and can be mutually singular and the nonlinear expectation in \cite{LR} is fully nonlinear, which do not preserve donimated convergence property.
	
	 In this paper, we aims to solve the multiple stopping problem under Knightian uncertainty, where Knightian uncertainty is modeled by $g$-expectation. It is well recognized that the $g$-expectatin first proposed by Peng \cite{P97} is a powerful tool to study problems facing Knightian uncertainty, since it preserves almost all properties of the classical expectation except the linearity by choosing some appropriate generator $g$ (see \cite{CR}, \cite{EPQ}). More precisely, if the generator $g$ is progressively measurable, Lipschitz continuous and concave, the $g$-expectation corresponds to the variational expectations. Therefore, the optiaml stopping problem under $g$-expectation is essentially a ``$\sup_\tau\inf_P$" problem.
	 
	  In fact, the optimal multiple stopping problem under Knightian uncertainty has been tackled by Li \cite{L}. However, to construct the optimal stopping times, the nonlinear expectation is assumed to be sub-additive and positive homogenous, which makes the optimal stopping problem a ``$\sup_\tau$$\sup_P$" problem. Compared with \cite{L}, the problem in this paper is more involved, which requires stronger regularity condition on the reward. Roughly speaking, the reward in \cite{L} is given by a set of random variables which are continuous along stopping times in expectation while the reward in the present paper is given by a continuous, progressively measurable process.
	 
	 In order to solve this problem, we construct a new reward process by induction, which is induced by the value function of an optimal single stopping problem. Then we obtain the associated new value functions as well as the optimal stopping times. By an induction method, the value function of the optimal multiple stopping problem is equivalent to the one corresponds to the single stopping problem with this new reward process. The optimal multiple stopping times can be constructed by this result easily. It is worth pointing out that, for the continuous-time case, the optimal stopping times exist when the reward process satisfies some regularity properties. This is the main difficulty in dealinig with the optimal multiple stopping problem. For this purpose, we need to apply the the fact that the value functions of the optimal stopping problem coincides with the solutions to reflected BSDEs as well as the fact the $g$-expectation is continuous in time. A natural observation is that the value function of the continuous time case is the limit of the ones in the discrete-time case, which would be useful to getting numerical approximations.
	 
	 We then investigate the properties of the optimal multiple stopping times for some typcial situations. Recall that if the reward process is a $g$-submartingale, the maturity time is optimal for the single stopping problem. Hence, it is easy to conjecture that the sequence of times $(T-(d-1)\delta, T-(d-2)\delta,\cdots, T)$ is the optimal multiple stopping times with $d$ exercise times and refracting period $\delta$. Second, due to the restriction that the difference between any two successive exercise times should be no less than $\delta$, one should exercise earlier if he has more than one exercise rights. These properties can be proved under some additional assumptions on the $g$-expectation (sublinear or superlinear). Third, suppose that the Knightian uncertainty is characterized by a special $g$-expectation, called $\kappa$-ignorance as in \cite{CE}, which corresponds to the generator $g(t,z)=-\kappa|z|$. When the reward process is given by a monotone increasing payoff function of a geometric Brownian motion, the value function coincides with the one of the classical case under the worst-case scenario with drfit $\kappa$. Similar results still holds for the decreasing case.
	 
	 This paper is organized as the following. In Section 2, we investigate the optimal multiple stopping problem for the discrete-time case. The continuous-time case is studied in Section 3. Section 4 is devoted to the properties of optimal stopping times and value functions of the optimal multiple stopping problem.   
		
	\section{Optimal multiple stopping: the discrete time case}
	In this section, we fix the time horizon $T\in\mathbb{N}$. Consider a filtered probability space $(\Omega,\mathcal{F},(\mathcal{F}_t)_{t\in[0,T]},P_0)$ satisfying the usual conditions in which the canonical process $B=\{B_t,0\leq t\leq T\}$ is a $d$-dimensional standard Brownian motion. $E^P[\cdot]$ is the expectation taken under $P$. When $P=P_0$, the expectation is usually denoted by $E[\cdot]$.
	
	Consider an investor who has $L$ times rights to exercise a contract under the constraints that he has to wait a minimal time $\delta\in\mathbb{N}$ before he exercises the next right, where $(L-1)\delta\leq T$. This assumption will exclude the case that the investor would exercise all the rights at the same time. When he exercises at time $i$, he will get the reward $X(i)$. The objective of the investor is trying to maximize the total expected reward. We first need to fix some technical condition on the payoff process.
	\begin{assumption}\label{a6}
		The payoff process $\{X(i)\}_{i=0}^T$ is nonnegative, adapted and satisfies the following condition:
		\begin{displaymath}
			E[\sum_{i=0}^{T}|X(i)|^2]<\infty.
		\end{displaymath}
	\end{assumption}
    
    Suppose that the investor faces Knightian uncertainty, that is, he will choose a set of beliefs $\mathcal{P}$ to evaluate his gain or loss. It is natural to assume that the investor is ambiguity-averse who measures the payoffs $\xi$  according to the variational preference given by 
    \begin{displaymath}
    	\mathcal{E}[\xi]=\inf_{P\in\mathcal{P}}(E^P[\xi]+c(P)),
    \end{displaymath}
    where $c(P)$ is a penalty function. El Karoui et al. \cite{EPQ} established that the $g$-expectation corresponds to the variational preference by choosing some appropriate generator $g$, where the (conditional) $g$-expectation $\mathcal{E}_t[\xi]$ of a square-integrable random variable $\xi$ is given by the solution of the following backward stochastic differential equation (BSDE):
    \begin{displaymath}
    	\mathcal{E}_t[\xi]=\xi+\int_t^T g(s,Z_s)ds-\int_t^T Z_sdB_s.
    \end{displaymath}
    For this purpose, we need to put the following assumptions on the generator $g$.
    
    \begin{assumption}\label{a7}
    	The generator $g:\Omega\times[0,T]\times \mathbb{R}^d\rightarrow\mathbb{R}$ satisfies the following conditions:
    	\begin{description}
    		\item[(1)] For any $z\in\mathbb{R}^d$, the process $\{g(t,z)\}_{t\in[0,T]}$ is progressively measurable with 
    		\begin{displaymath}
    			E[\int_0^T |g(t,z)|^2dt]<\infty;
    		\end{displaymath}
    		\item[(2)] There exists a constant $\kappa>0$ such that for any $z,z'\in\mathbb{R}^d$ and $t\in[0,T]$, 
    		\begin{displaymath}
    			|g(t,z)-g(t,z')|\leq \kappa|z-z'|;
    		\end{displaymath}
    		\item[(3)] $g$ is concave in $z$ and $g(t,0)=0$ for any $t\in[0,T]$.
    	\end{description}
    \end{assumption}

    Indeed, let $f$ be the convex dual of $g$, i.e.,
    \begin{displaymath}
    	f(\omega,t,\theta)=\sup_{z\in\mathbb{R}^d}(g(\omega,t,z)-z\cdot\theta).
    \end{displaymath}
    We denote by $D$ the collection of all progressively measurable processes $\{\theta_t\}_{t\in[0,T]}$ with 
    \begin{displaymath}
    	E[\int_0^T |f(s,\theta_s)|^2ds]<\infty.
    \end{displaymath}
    For each fixed $\theta\in D$, we may construct a probability measure $P_\theta$ equivalent to $P_0$ with density
    \begin{displaymath}
    	\frac{dP_\theta}{dP_0}=\exp(\int_0^T\theta_tdB_t-\frac{1}{2}\int_0^T\theta_t^2dt).
    \end{displaymath}
    Then, for any stopping time $\tau\in[t,T]$ and an $\mathcal{F}_\tau$-measurable and square-integrable random variable $\xi$, the conditional $g$-expectation can be represented as the following:
    \begin{displaymath}
    	\mathcal{E}_t[\xi]=\esssup_{\theta\in D}(E^\theta_t[\xi]+E^\theta_t[\int_t^\tau f(s,\theta_s)ds]),
    \end{displaymath}
    where $E^\theta_t[\cdot]$  is the conditional expectation taken under probability $P_\theta$. Hence, the $g$-expectation coincides with the variational preference when $g$ satisfies Assumption \ref{a7}.
    
    Now we formulate the multiple stopping problem under $g$-expectation. For notional convenience, for any process $\{\eta(i)\}_{i=0}^T$, set $\eta(i)=0$ and $\mathcal{F}_i=\mathcal{F}_T$ for $i>T$. For $L\geq 2$, let $\mathcal{S}_i(L,\delta)$ be the set of all vectors of stopping times $\tau^{(L)}(i)=(\tau_1(i),\cdots,\tau_L(i))$ such that $i\leq \tau_1(i)$ and, for any $2\leq j\leq L$, $\tau_{j-1}(i)+\delta\leq \tau_j(i)$. If $L=1$, then $\mathcal{S}_i(1,\delta)=\mathcal{S}_i$, the collection of all stopping times no less than $i$. The value function of the optimal multiple stopping problem can be stated as follows:
    \begin{equation}\label{e2}
    	Z_L(i)=\esssup_{\tau^{(L)}(i)\in\mathcal{S}_i(L,\delta)}\mathcal{E}_i[\sum_{j=1}^L X(\tau_j(i))].
    \end{equation}
    Then, the value function has the following representation.
    
    \begin{lemma}\label{l6}
    	The value function of the multiple stopping problem can be obtained recursively: for any $j=2,\cdots,L$, $Z_j(T)=X_T$ and 
    	\begin{displaymath}
    	Z_j(i)=\max\{X(i)+\mathcal{E}_i[Z_{j-1}(i+\delta)], \mathcal{E}_i[Z_j(i+1)]\}, \ i=0,\cdots,T-1,
    	\end{displaymath}
    	where $Z_1$ is the value function of the optimal single stopping problem with reward sequence $\{X(i)\}_{i=1}^T$.
    \end{lemma}
    
    \begin{proof}
    	It is easy to check the terminal condition holds. For any $\tau^{(j)}(i)=(\tau_1(i),\cdots,\tau_j(i))$ with $i\leq T-1$, on the set $A=\{\tau_1(i)=i\}\in\mathcal{F}_i$, we have $(\tau_2(i),\cdots,\tau_j(i))\in\mathcal{S}_{i+\delta}(j-1,\delta)$ and 
    	\begin{align*}
    	\mathcal{E}_i[\sum_{l=1}^jX(\tau_l(i))]I_A&=\mathcal{E}_i[\{X(i)+\sum_{l=2}^jX(\tau_l(i))\}I_A]=X(i)I_A+\mathcal{E}_i[\mathcal{E}_{i+\delta}[\sum_{l=2}^jX(\tau_l(i))]]I_A\\ &\leq \{X(i)+\mathcal{E}_i[Z_{j-1}(i+\delta)]\}I_A.
    	\end{align*}
    	On the set $A^c$, we obtain that $\tau^{(j)}(i)\in\mathcal{S}_{i+1}(j,\delta)$. It follows that 
    	\begin{displaymath}
    	\mathcal{E}_i[\sum_{l=1}^jX(\tau_l(i))]I_{A^c}=\mathcal{E}_i[\mathcal{E}_{i+1}[\sum_{l=1}^jX(\tau_l(i))]I_{A^c}]\leq \mathcal{E}_i[Z_j(i+1)]I_{A^c}.
    	\end{displaymath}
    	Combining the above two equations yields that 
    	\begin{displaymath}
    	\mathcal{E}_i[\sum_{l=1}^jX(\tau_l(i))]\leq \max\{X(i)+\mathcal{E}_i[Z_{j-1}(i+\delta)],\mathcal{E}_i[Z_j(i+1)]\}.
    	\end{displaymath}
    	Since $\tau^{(j)}(i)$ is chosen arbitrarily, we deduce that 
    	\begin{displaymath}
    	Z_j(i)\leq  \max\{X(i)+\mathcal{E}_i[Z_{j-1}(i+\delta)],\mathcal{E}_i[Z_j(i+1)]\}.
    	\end{displaymath}
    	
    	We are now in a position to show the inverse inequality. We claim that, for any $j=1,\cdots,L$ and $i=0,\cdots,T$, there exists a sequence $\{(\tau_1^n(i),\cdots,\tau_j^n(i))\}_{n\in\mathbb{N}}\subset \mathcal{S}_i(j,\delta)$ such that $\mathcal{E}_i[\sum_{l=1}^j X(\tau_l^n(i))]$ converges to $Z_j(i)$ upward. Indeed, it is sufficient to prove that, for any $(\tau_1^{m}(i),\cdots,\tau_j^m(i))\in\mathcal{S}_i(j,\delta)$, $m=1,2$, we can find $(\tau_1(i),\cdots,\tau_j(i))\in\mathcal{S}_i(j,\delta)$ such that 
    	\begin{displaymath}
    	\mathcal{E}_i[\sum_{l=1}^j X(\tau_l(i))]=\max\{	\mathcal{E}_i[\sum_{l=1}^j X(\tau^1_l(i))],	\mathcal{E}_i[\sum_{l=1}^j X(\tau^2_l(i))]\}.
    	\end{displaymath}
    	For this purpose, set $B=\{	\mathcal{E}_i[\sum_{l=1}^j X(\tau^1_l(i))]\geq 	\mathcal{E}_i[\sum_{l=1}^j X(\tau^2_l(i))]\}\in\mathcal{F}_i$ and $\tau_l(i)=\tau^1_l(i)I_B+\tau^2_l(i)I_{B^c}$, $l=1,\cdots,j$. It is easy to check that $(\tau_1(i),\cdots,\tau_j(i))\in\mathcal{S}_i(j,\delta)$ and 
    	\begin{displaymath}
    	\mathcal{E}_i[\sum_{l=1}^j X(\tau_l(i))]=\mathcal{E}_i[\sum_{l=1}^j X(\tau^1_l(i))]I_B+\mathcal{E}_i[\sum_{l=1}^j X(\tau^2_l(i))] I_{B^c}=\max\{	\mathcal{E}_i[\sum_{l=1}^j X(\tau^1_l(i))],	\mathcal{E}_i[\sum_{l=1}^j X(\tau^2_l(i))]\}.
    	\end{displaymath}
    	Hence, the claim follows. Then, noting that $\mathcal{S}_{i+1}(j,\delta)\subset \mathcal{S}_i(j,\delta)$, by monotone convergence theorem, we have 
    	\begin{displaymath}
    	\mathcal{E}_i[Z_j(i+1)]=\mathcal{E}_i[\lim_{n\rightarrow\infty}\mathcal{E}_{i+1}[\sum_{l=1}^j X(\tau_l^n(i+1))]]=\lim_{n\rightarrow\infty}\mathcal{E}_i[\mathcal{E}_{i+1}[\sum_{l=1}^j X(\tau_l^n(i+1))]]\leq Z_j(i).
    	\end{displaymath}
    	Note that $(i,\tau_1^n(i+\delta),\cdots,\tau_{j-1}^n(i+\delta))\in\mathcal{S}_i(j,\delta)$ if $(\tau_1^n(i+\delta),\cdots,\tau_{j-1}^n(i+\delta))\in \mathcal{S}_{i+\delta}(j-1,\delta)$. Similar analysis as above indicates that $X(i)+\mathcal{E}_i[Z_{j-1}(i+\delta)]\leq Z_j(i)$. The proof is complete.
    \end{proof}

	In order to solve the optimal multiple stopping problem \eqref{e2}, motivated by Lemma \ref{l6}, we introduce the following family of optimal stopping problems with one exercise right:
	\begin{equation}\label{e3}
		Y^{(0)}\equiv 0 \textrm{ and } Y^{(j)}(i)=\esssup_{\tau\in \mathcal{S}_{i}}\mathcal{E}_i[X^{(j)}(\tau)],
	\end{equation}
	where $X^{(j)}(i)=X(i)+\mathcal{E}_i[Y^{(j-1)}(i+\delta)]$, for any $i=0,\cdots,T$ and $j=1,2,\cdots, L$. Our first observation is that the process $\{X^{(j)}(i)\}_{i=0}^T$ satisfies Assumption \ref{a6}, for any $j=1,\cdots,L$.
	\begin{lemma}\label{l2}
		For any $j=1,\cdots,L$, the process $\{X^{(j)}(i)\}_{i=0}^T$ is nonnegative, adapted and satisfies the following condition:
		\begin{displaymath}
		E[\sum_{i=0}^{T}|X^{(j)}(i)|^2]<\infty.
		\end{displaymath}
	\end{lemma}
    
    \begin{proof}
      We prove this result by induction. Since $X^{(1)}=X$, the statement holds trivially for $j=1$. Assume that the results hold true for any $j\leq k$. It remains to show that $E[|X^{(k+1)}(i)|^2]<\infty$, for any $i=0,\cdots,T$. It is easy to check that 
      \begin{displaymath}
      	Y^{(k)}(i)\leq \mathcal{E}_i[\sum_{j=0}^T X^{(k)}(j)].
      \end{displaymath} 
      By the estimate of BSDE, we have
      \begin{displaymath}
      	E[|Y^{(k)}(i)|^2]\leq E[|\mathcal{E}_i[\sum_{j=0}^T X^{(k)}(j)]|^2]\leq C E[\sum_{j=0}^T |X^{(k)}(j)|^2]<\infty.
      \end{displaymath}
      Applying the estimate for BSDE again yields that 
      \begin{align*}
      	E[|X^{(k+1)}(i)|^2]&\leq C\{E[|X(i)|^2]+E[|\mathcal{E}_i[Y^{(k)}(i+\delta)]|^2]\}\\ &\leq C\{E[|X(i)|^2]+E[|Y^{(k)}(i+\delta)|^2]\}<\infty.
      \end{align*}
    \end{proof}

    By the result of Riedel \cite{R}, for any $j=1,\cdots,L$, $(Y^{(j)}(i):0\leq i\leq T)$ admits the following backward representation: $Y^{(j)}(T)=X^{(j)}(T)$ and for any $i=0,\cdots,T-1$
    \begin{equation}\label{e30}
    	Y^{(j)}(i)=\max\{X^{(j)}(i),\mathcal{E}_i[Y^{(j)}(i+1)]\}.
    \end{equation}
    Besides, the following stopping times
    \begin{equation}\label{e5}
    	\sigma^*_j(i)=\inf\{i\leq l: Y^{(j)}(l)=X^{(j)}(l)\}=\inf\{i\leq l: X^{(j)}(l)\geq \mathcal{E}_l[Y^{(j)}(l+1)]\}
    \end{equation}
    are optimal for the optimal stopping problem \eqref{e3}, that is, $Y^{(j)}(i)=\mathcal{E}_i[X^{(j)}(\sigma^*_j(i))]$. 
     If $i>T$, we define $\sigma_j^*(i)=i$. Now let us set 
    \begin{equation}\begin{split}\label{e}
    \tau^*_{1,j}(i)&=\sigma^*_j(i),\\
    \tau^*_{d+1,j}(i)&=\tau^*_{d,j-1}(\sigma^*_{j}(i)+\delta), \ 1\leq d\leq L-1.
    \end{split}
    \end{equation}
   Then, we have the following characterization of the value function of the auxiliary optimal stopping problem \eqref{e3}.
    \begin{lemma}\label{l3}
    	Let Assumption \ref{a6} hold. Then, we have for any $i=0,1,\cdots,T$, $j=1,\cdots,L$,
    	\begin{equation}\label{e4}
    		Y^{(j)}(i)=\mathcal{E}_i[\sum_{l=1}^j X(\tau^*_{l,j}(i))],
    	\end{equation}
    	and for any $d=2,\cdots,j$, $\tau^*_{d,j}(i)-\tau^*_{d-1,j}(i)\geq \delta$.
    \end{lemma}

    \begin{proof}
    	It is obvious that Equation \eqref{e4} holds for the case $j=1$. Assume that the statement holds true for any $j=n$, i.e., 
    	\begin{displaymath}
    		Y^{(n)}(i)=\mathcal{E}_i[\sum_{l=1}^n X(\tau^*_{l,n}(i))].
    	\end{displaymath}
    	By simple calculation, we have
    	\begin{align*}
    		Y^{(n+1)}(i)&=\mathcal{E}_i[X^{(n+1)}(\sigma^*_{n+1}(i))]=\mathcal{E}_i[X(\sigma^*_{n+1}(i))+Y^{(n)}(\sigma^*_{n+1}(i)+\delta)]\\
    		&=\mathcal{E}_i[X(\sigma^*_{n+1}(i))+\mathcal{E}_{\sigma^*_{n+1}(i)+\delta}[\sum_{l=1}^n X(\tau^*_{l,n}(\sigma^*_{n+1}(i)+\delta))]]\\
    		&=\mathcal{E}_i[X(\tau^*_{1,n+1}(i))+\sum_{l=1}^n X(\tau^*_{l+1,n+1}(i))]=\mathcal{E}_i[\sum_{l=1}^{n+1} X(\tau^*_{l,n+1}(i))].
    	\end{align*}
    	
    	Now we claim that for any $d=2,\cdots,j$, $\tau^*_{d,j}(i)=\sigma^*_{j-d+1}(\tau^*_{d-1,j}(i)+\delta)$, which implies that $\tau^*_{d,j}(i)-\tau^*_{d-1,j}(i)\geq \delta$. We will prove this claim by induction over $d$. Indeed, if $d=2$, it is easy to check that, for any $j\geq 2$, 
    	\begin{displaymath}
    		\tau^*_{2,j}(i)=\tau^*_{1,j-1}(\sigma^*_j(i)+\delta)=\sigma^*_{j-1}(\tau^*_{1,j}(i)+\delta).
    	\end{displaymath}
    	Assume that for any $d=k$, we have $\tau^*_{k,j}(i)=\sigma^*_{j-k+1}(\tau^*_{k-1,j}(i)+\delta)$, for any $j\geq k$. For the case $d=k+1$, by simple calculation, we obtain that, for any $j\geq k+1$,
    	\begin{displaymath}
    		\tau^*_{k+1,j}(i)=\tau^*_{k,j-1}(\sigma^*_j(i)+\delta)=\sigma^*_{j-k}(\tau^*_{k-1,j-1}(\sigma^*_j(i)+\delta)+\delta)=\sigma^*_{j-k}(\tau^*_{k,j}(i)+\delta).
    	\end{displaymath}
    	Hence the claim follows.
    \end{proof}
    
    By Lemma \ref{l6}, Equation \eqref{e30} and noting that $Y^{(1)}(i)=Z_1(i)$ for any $i=0,1,\cdots,T$, we have 
    \begin{displaymath}
    Y^{(j)}(i)=Z_j(i), \ j=1,2,\cdots,L, \ i=0,1,\cdots,T.
    \end{displaymath} Therefore, the stopping times $(\tau^*_{1,j}(i),\cdots,\tau^*_{j,j}(i))$ are optimal for the  multiple stopping problem \eqref{e2} according to Lemma \ref{l3}
    \begin{theorem}\label{t1}
    	Let Assumption \ref{a6} hold. Then, for any $i=0,\cdots,T$ and $j=1,\cdots,L$, we have
    	\begin{displaymath}
    		Z_j(i)=Y^{(j)}(i)=\mathcal{E}_i[\sum_{l=1}^j X(\tau^*_{l,j}(i))].
    	\end{displaymath}
    \end{theorem}

	

\begin{remark}\label{r1}
	By Theorem \ref{t1} and Equation \eqref{e5}, we may obtain that 
	\begin{displaymath}
		\sigma^*_j(i)=\inf\{i\leq l: X(l)+\mathcal{E}_l[Z_{j-1}(l+\delta)]\geq \mathcal{E}_l[Z_j(l+1)]\}.
	\end{displaymath}
\end{remark}



	\section{Optimal multiple stopping: the continuous-time case}
	
	In this section, we focus on the optimal multiple stopping problem for the continuous-time case. The evaluation $\mathcal{E}[\cdot]$ remains the same with the one in Section 2. We need to propose the following assumption on the payoff process $\{X_t\}_{t\in[0,T]}$.	
	
	\begin{assumption}\label{a1}
		The payoff process $X=\{X_t\}_{t\in[0,T]}$ is an adapted, nonnegative process with $X_T=0$, and it has continuous sample path satisfying the following property:
		\begin{displaymath}
			E[\bar{X}^2]<\infty,
		\end{displaymath}
		where $\bar{X}=\sup_{t\in[0,T]}X_t$.
	\end{assumption}

    For the convenience of analysis, for any process $\xi$, we define $\xi_s=0$, $s> T$ and for any constant $s$, any integers $n<m$ and process $\eta$, we define $\sum_{i=s+m}^{s+n}\eta_i=0$. Now we fix an integer $L\geq 2$ and a positive constant $\delta$ such that $(L-1)\delta\leq T$. Consider a contract that one has at most $L$ rights to exercise under the restriction that the difference between any two successive exercise times is no less than $\delta$. This assumption on the seperation of the exercise times avoids the possibility to exercise the contract simutaneously. Let $\bar{\mathcal{S}}_t(L,\delta)$ be the collection of all vectors of stopping times $\tau^{(L)}(t)=(\tau_1(t),\cdots,\tau_L(t))$ satisfying
    \begin{displaymath}
    	\tau_1(t)\geq t,  \tau_i(t)-\tau_{i-1}(t)\geq \delta \textrm{ for all } i=2, \cdots,L.
    \end{displaymath}
    For the case that $L=1$, $\bar{\mathcal{S}}_t(1,\delta)=\bar{\mathcal{S}}_t$ represents the collection of all stopping times no less than $t$. The value function of the optimal multiple stopping problem is defined by
    \begin{equation}\label{e1}
      Z_t^L=\esssup_{\tau^{(L)}(t)\in\bar{\mathcal{S}}_t(L,\delta)}\mathcal{E}_t[\sum_{j=1}^L X_{\tau_j(t)}].	
    \end{equation}
   In fact, the holder of this contract may not use all the exercise rights due to the refracing period $\delta>0$. She may sacrifice several right in the hope that exercise this contract in the future may be more beneficial.  
    
    \begin{remark}
    	In fact, if the holder has $L$ times exercise rights and $X_t>0$ for any $t\in[0,T)$, she will sacrifice at most $[\frac{L}{2}]$ times rights. Indeed, if $L=2n$ and the holder only exercise $(n-1)$ times rights, let us denote the exercise times by $(\tau_1,\cdots,\tau_{n-1})$. Then only the following three cases may occur: (1) $\tau_1\geq \delta$; (2) there exists some $2\leq k\leq n-1$ such that $\tau_k-\tau_{k-1}\geq 2\delta$; (3) $\tau_1<\delta$ and $\tau_k-\tau_{k-1}<2\delta$ for any $2\leq k\leq n-1$. If (1) happens, set $\tau=\tau_1-\delta$. If (2) happens, set $\tau=\tau_k-\delta$. If (3) happens, set $\tau=\tau_{n-1}+\delta$. Noting that $\delta(L-1)\leq T$, it is easy to check that $\tau=\tau_1+\sum_{k=2}^{n-1}(\tau_k-\tau_{k-1})+\delta<\delta+2\delta(n-2)+\delta=2n\delta-2\delta\leq T$. Then under all these three cases, we have $\mathcal{E}[\sum_{i=1}^{n-1}X_{\tau_i}]<\mathcal{E}[\sum_{i=1}^{n-1}X_{\tau_i}+X_{\tau}]$, which is a contradiction. Similar analysis still holds if $L$ is odd.
    \end{remark}
    
     In order to solve the original problem \eqref{e1}, motivated by the analysis in the discrete-time case, we need to introduce the following auxiliary optimal single  stopping problems:
    \begin{equation}\label{e01}
    	Y^{(0)}\equiv 0 \textrm{ and } Y_t^{(i)}=\widehat{X^{(i)}}_t:=\esssup_{\tau\in \bar{\mathcal{S}}_{t}}\mathcal{E}_t[X^{(i)}_\tau], \textrm{ for } i=1,2,\cdots,L,
    \end{equation}
    where, for each $i=1,2,\cdots,L$, the $i$-exercise reward process $X^{(i)}$ is defined by:
    \begin{displaymath}
    	X_t^{(i)}=\begin{cases}
    	X_t+\mathcal{E}_t[Y^{(i-1)}_{t+\delta}], &t\in[0,T-\delta];\\
    	X_t, &t\in(T-\delta,T].
    	\end{cases}
    \end{displaymath}
    
    \begin{remark}
    	Recall the convention that $\xi_s=0$ for any process $\xi$ when $s> T$. The reward process $X^{(i)}$ can be rewritten as $X_t^{(i)}=X_t+\mathcal{E}_t[Y^{(i-1)}_{t+\delta}]$, $t\in[0,T]$.
    \end{remark} 

    As in Section 2, the optimal stopping times of problem \eqref{e1} is derived from the ones of \eqref{e01}. The existence of optimal stopping times for the auxiliary problems \eqref{e01} requires some regularity of the reward process $X^{(i)}$. In fact, we can show that, the new reward process $X^{(i)}$ preserves all properties of the original reward process $X$. For this purpose, we first introduce the following typical $g$-expectation which dominates $\mathcal{E}[\cdot]$.  For any $\mathcal{F}_T$-measurable and square integrable random variable $\xi$, we define 
    \begin{displaymath}
    	\mathcal{E}^\kappa_t[\xi]=Y^\xi_t,
    \end{displaymath}
    where $(Y^\xi,Z^\xi)$ is the solution of the following BSDE:
    \begin{displaymath}
    	Y_t^\xi=\xi+\int_t^T \kappa|Z_s^\xi|ds-\int_t^T Z^\xi_s dB_s.
    \end{displaymath}
    By the results in \cite{CHMP} and \cite{P97}, $\mathcal{E}^\kappa_t[\cdot]$ is sublinear and dominates $\mathcal{E}_t[\cdot]$, i.e., for any $\xi,\eta\in L^2(P_0)$, we have
    \begin{displaymath}
    	|\mathcal{E}_t[\xi]-\mathcal{E}_t[\eta]|\leq \mathcal{E}^\kappa_t[|\xi-\eta|].
    \end{displaymath}
   
    \begin{lemma}\label{l1}
    	Suppose that $X$ satisfies Assumption \ref{a1}. Then, for all $i=1,2,\cdots,L$, the process $X^{(i)}$ is continuous, and satisfies
    	\begin{displaymath}
    		E[\overline{X^{(i)}}^2]<\infty,
    	\end{displaymath}
    	where $\overline{X^{(i)}}=\sup_{t\in[0,T]}X_t^{(i)}$.
    \end{lemma}

    \begin{proof}
    	We need to prove this result by induction. It is easy to check that the statement holds for the case $i=1$ since $X^{(1)}=X$. Now we assume that for $i\geq 2$, the process $X^{(i-1)}$ is continuous, and satisfies $E[\overline{X^{(i-1)}}^2]<\infty$. We first show that $\overline{X^{(i)}}$ is square integrable. Note that 
    	\begin{align*}
    		&Y^{(i-1)}_t\leq \mathcal{E}_t[\overline{X^{(i-1)}}],\\
    		&\overline{X^{(i)}}\leq \bar{X}+\sup_{t\in[0,T]}\mathcal{E}_t[Y^{(i-1)}_{t+\delta}]\leq \bar{X}+\sup_{t\in[0,T]}\mathcal{E}_t[\sup_{s\in[0,T]} Y^{(i-1)}_s].
    	\end{align*} 
    	By the estimate for BSDE, we have
    	\begin{displaymath}
    		E[\sup_{t\in[0,T]} (Y_t^{(i-1)})^2]\leq E[\sup_{t\in[0,T]}|\mathcal{E}_t[\overline{X^{(i-1)}}]|^2]\leq C E[\overline{X^{(i-1)}}^2]<\infty.
    	\end{displaymath}
    	It follows that 
    	\begin{displaymath}
    		E[\overline{X^{(i)}}^2]\leq C\{E[\bar{X}^2]+E[\sup_{t\in[0,T]}(\mathcal{E}_t[\sup_{s\in[0,T]} Y^{(i-1)}_s])^2]\}\leq C\{E[\bar{X}^2]+E[\sup_{t\in[0,T]} (Y_t^{(i-1)})^2]\}<\infty,
    	\end{displaymath}
    	where we have used the estimate for BSDE again in the second inequality. 
    	
    	Now we prove the process $X^{(i)}$ is continuous, which is equivalent to show that $\{\mathcal{E}_t[Y^{(i-1)}_{t+\delta}]\}_{t\in[0,T]}$ is continuous. For simplicity, we omit the superscript $(i-1)$. By Corollary D.1 in \cite{CR}, there exists a pair of processes $(Z,K)$ such that $(Y,Z,K)$ is the solution of reflected BSDE with terminal value $X_T$, generator $g$ and obstacle process $X$, i.e., $(Y,Z,K)$ satisfies the following conditions
    	\begin{description}
    		\item[(1)] $Y_t=X_T+\int_t^T g(s,Z_s)ds-\int_t^T Z_sdB_s+K_T-K_t$;
    		\item[(2)] $Y_t\geq X_t$, $t\in[0,T]$ and $\int_0^T (Y_t-X_t)dK_t=0$,
    	\end{description}
    where $Z$ is an adapted process with $E[\int_0^T Z_t^2 dt]<\infty$, $K$ is continuous and increasing, $K_0=0$ and $K_T$ is square integrable. Therefore, the process $Y$ is continuous. It is sufficient to prove that the process $\{\mathcal{E}_t[Y_{t+\delta}]\}_{t\in[0,T]}$ is left continuous. Let $\{t_n\}_{n\in\mathbb{N}}$ be a sequence of times such that $t_n\uparrow t$. It is easy to check that 
    \begin{align*}
    	|\mathcal{E}_t[Y_{t+\delta}]-\mathcal{E}_{t_n}[Y_{t_n+\delta}]|&\leq |\mathcal{E}_t[Y_{t+\delta}]-\mathcal{E}_{t_n}[Y_{t+\delta}]|+|\mathcal{E}_{t_n}[Y_{t+\delta}]-\mathcal{E}_{t_n}[Y_{t_n+\delta}]|\\
    	&\leq |\mathcal{E}_t[Y_{t+\delta}]-\mathcal{E}_{t_n}[Y_{t+\delta}]|+\mathcal{E}^\kappa_{t_n}[|Y_{t+\delta}-Y_{t_n+\delta}|]\\
    	&\leq |\mathcal{E}_t[Y_{t+\delta}]-\mathcal{E}_{t_n}[Y_{t+\delta}]|+\mathcal{E}^\kappa_{t_n}[\xi_m]:=I+II,
    \end{align*}
    where $\xi_m=\sup_{r\geq m}|Y_{t+\delta}-Y_{t_r+\delta}|$ with $m\leq n$. Note that $\{\mathcal{E}_s[Y_{t+\delta}]\}_{s\in[0,T]}$ can be viewed as the first component of the solution to the BSDE with terminal value $Y_{t+\delta}$ and generator $g$. Hence, it is continuous, which implies that $\lim_{n\rightarrow\infty} I=0$. By the previous analysis, we have for any $m\leq n$,
    \begin{displaymath}
    	E[|\xi_m|^2]\leq C E[\sup_{t\in[0,T]}Y_t^2]<\infty.
    \end{displaymath}
    Therefore, $\{\mathcal{E}^\kappa_s[\xi_m]\}_{s\in[0,T]}$ coincides with the first component of the solution to the BSDE with terminal value $\xi_m$ and generator $g(t,z)=\kappa|z|$. It follows from the continuity of the solution that $\lim_{n\rightarrow\infty}II=\mathcal{E}^\kappa_t[\xi_m]$. It remains to prove that $\lim_{m\rightarrow\infty}\mathcal{E}^\kappa_t[\xi_m]=0$. Indeed, we have
    \begin{align*}
    	\mathcal{E}^\kappa_t[\xi_m]&\leq \mathcal{E}^\kappa_t[\int_{t_m+\delta}^{t+\delta}|g(s,Z_s)|ds]+\mathcal{E}^\kappa_t[\sup_{r\geq m}|\int_{t_r+\delta}^{t+\delta} Z_sdB_s|]+\mathcal{E}^\kappa_t[\sup_{r\geq m}|K_{t+\delta}-K_{t_r+\delta}|]\\
    	&\leq C\{(E_t[(\int_{t_m+\delta}^{t+\delta}|g(s,Z_s)|ds)^2])^{1/2}+(E_t[\sup_{r\geq m}|\int_{t_r+\delta}^{t+\delta} Z_sdB_s|^2])^{1/2}+(E_t[|K_{t+\delta}-K_{t_m+\delta}|^2])^{1/2}\},
    \end{align*}
    where we use the estimate for BSDE and the fact that $K$ is increasing in the second inequality. By the B-D-G inequality and dominated convergence theorem, we get the desired result. The proof is complete.
    \end{proof}

    Recall that a progressively measurable process $\{\xi_t\}_{t\in[0,T]}$ is said to be a $g$-supermartingale on $[0,T]$ in strong sense if, for each stopping time $\tau\leq T$, $E[|\xi_\tau|^2]<\infty$, and for all stopping time $\sigma\leq \tau$, we have $\mathcal{E}_\sigma[\xi_\tau]\leq \xi_\sigma$. A progessively measurable process $\{\xi_t\}_{t\in[0,T]}$ is said to be a $g$-supermartingale on $[0,T]$ in weak sense if, for each deterministic time $t\leq T$, $E[|\xi_t|^2]<\infty$, and for all deterministic time $s\leq t$, we have $\mathcal{E}_s[\xi_t]\leq \xi_s$. It is natural to conclude that a $g$-supermartingale in strong sense is a $g$-supermartingale in weak sense. Furthermore, a $g$-supermartingale in weak sense is also a $g$-supermartingale in strong sense provided that it is right-continuous. For more details, we may refer to the papers \cite{CP} and \cite{P99}. Note that in the proof of Lemma \ref{l1}, the process $\{Y^{(i)}_t\}_{t\in[0,T]}$ is continuous. Then, by Theorem 3.1 in \cite{CR}, we may conclude that for any $i=1,\cdots,L$,
    \begin{description}
    	\item[(1)] $Y^{(i)}$ is the smallest $g$-supermartingale in strong sense dominating the reward process $X^{(i)}$;
    	\item[(2)] $\sigma_i^*(t)=\inf\{s\geq t: Y^{(i)}_s=X^{(i)}_s\}$ is an optimal stopping time;
    	\item[(3)] the value function stopped at time $\sigma_i^*(t)$, $\{Y_{s\wedge \sigma_i^*(t)}\}_{s\in[t,T]}$ is a $g$-martingale in strong sense.
    \end{description}

    Now consider the following stopping times: $\tau^*_{1,L}(t)=\inf\{s\geq t: Y_s^{(L)}=X_s^{(L)}\}$ and for $i=2,\cdots,L$, set
    \begin{equation}\label{e31}\begin{split}
    	\tau_{i,L}^*(t)=&\inf\{s\geq \delta+\tau_{i-1,L}^*(t):Y^{(L-i+1)}_s=X^{(L-i+1)}_s\}I_{\{\delta+\tau^*_{i-1,L}(t)\leq T\}}\\ &+(\delta+\tau^*_{i-1,L}(t))I_{\{\delta+\tau^*_{i-1,L}(t)> T\}}.
    	\end{split}
    \end{equation}
    It is easy to check that ${\tau^{(L)*}(t)}=(\tau^*_{1,L}(t),\cdots,\tau_{L,L}^*(t))\in \bar{\mathcal{S}}_t(L,\delta)$ and the representation similar with Equation \eqref{e} still holds. We then show that ${\tau^{(L)*}(t)}$ is optimal for the multiple stopping problem \eqref{e1}. The proof is motivated by the one of Theorem 1 in \cite{CT}.
    
    \begin{theorem}\label{t}
    	If the process $X$ satisfies Assumption \ref{a1}, then we have for any $t\in[0,T]$
    	\begin{displaymath}
    		Z^L_t=Y_t^{(L)}=\mathcal{E}_t[\sum_{i=1}^L X_{\tau_{i,L}^*(t)}].
    	\end{displaymath}
    \end{theorem}

   \begin{proof}
   	We only proof the case that $t=0$, since the other cases can be proved in a similar way. For simplicity, we omit $t$ in the parentheses.  Let $(\tau_1,\cdots,\tau_L)\in\bar{\mathcal{S}}_0(L,\delta)$. We first claim that, for any $n=1,2,\cdots,L$, the following statement holds:
   	\begin{equation}\label{a2}
   		X^{(n)}_{\tau_{L-n+1}}\geq \mathcal{E}_{\tau_{L-n+1}}[\sum_{i=L-n+1}^L X_{\tau_i}]
   	\end{equation}
   	Note that $X^{(1)}\geq X$, which implies that Equation \eqref{a2} holds for $n=1$. Assume that Equation \eqref{a2} holds for $n=k$. On the set $\{\tau_{L-k}\leq T-\delta\}$, we obtain that
   	\begin{align*}
   		X^{(k+1)}_{\tau_{L-k}}&=\mathcal{E}_{\tau_{L-k}}[X_{\tau_{L-k}}+Y_{\tau_{L-k}+\delta}^{(k)}]\geq \mathcal{E}_{\tau_{L-k}}[X_{\tau_{L-k}}+\mathcal{E}_{\tau_{L-k}+\delta}[X^{(k)}_{\tau_{L-k+1}}]]\\
   		&\geq \mathcal{E}_{\tau_{L-k}}[X_{\tau_{L-k}}+\mathcal{E}_{\tau_{L-k}+\delta}[\mathcal{E}_{\tau_{L-k+1}}[\sum_{i=L-k+1}^L X_{\tau_i}]]]=\mathcal{E}_{\tau_{L-k}}[\sum_{i=L-k}^L X_{\tau_i}].
   	\end{align*}
   	On the set $\{\tau_{L-k}>T-\delta\}$, then for any $i\geq L-k+1$, we have $\tau_i>T$, which implies that $X_{\tau_i}=0$. It follows that
   	\begin{displaymath}
   		\mathcal{E}_{\tau_{L-k}}[\sum_{i=L-k}^L X_{\tau_i}]= X_{\tau_{L-k}}=X^{(k+1)}_{\tau_{L-k}}.
   	\end{displaymath} 
   	Hence, the claims holds.
   	
   	Now, for simplicity, we set $\bar{\tau}_n:=\tau_{L-n+1}$ and $X^{(0)}\equiv 0$. By Equation \eqref{a2}, it is easy for us to check that, for any $n=0,1,\cdots,L$, we have
   	\begin{equation}\label{a3}
   		\mathcal{E}[\sum_{j=1}^L X_{\tau_j}]\leq \mathcal{E}[X^{(n)}_{\bar{\tau}_n}+\sum_{j=n+1}^L X_{\bar{\tau}_j}].
   	\end{equation}
   	Indeed, if $n=0$, then, two sides of the above equation are equal. For the case that $n= 1,\cdots,L$, by simple calculation and Equation \eqref{a2}, we obtain that 
   	\begin{displaymath}
   		\mathcal{E}[\sum_{j=1}^L X_{\tau_j}]=\mathcal{E}[\sum_{i=1}^{L-n} X_{\tau_i}+\mathcal{E}_{\tau_{L-n+1}}[\sum_{i=L-n+1}^L X_{\tau_i}]]\leq \mathcal{E}[X^{(n)}_{\bar{\tau}_n}+\sum_{j=n+1}^L X_{\bar{\tau}_j}].
   	\end{displaymath}
   	Taking $n=L$, together with the classical optimal stopping problem and the definition of $X^{(L)}$, yields that
   	\begin{equation}\label{a4}
   		\mathcal{E}[\sum_{j=1}^L X_{\tau_j}]\leq \mathcal{E}[X^{(L)}_{\tau_1}]\leq Y_0^{(L)}=\mathcal{E}[X^{(L)}_{\tau^*_{1,L}}]=\mathcal{E}[X_{\tau^*_{1,L}}+\mathcal{E}_{\tau^*_{1,L}}[Y_{\delta+\tau^*_{1,L}}^{(L-1)}]]=\mathcal{E}[X_{\tau^*_{1,L}}+Y_{\delta+\tau^*_{1,L}}^{(L-1)}].
   	\end{equation}
   	Note that the stopped process $\{Y^{(L-1)}_{t\wedge \tau^*_{2,L}};t\geq \delta+\tau^*_{1,L}\}$ is a $g$-martingale, which implies that 
   	\begin{equation}\label{a5}
   		Y_{\delta+\tau_{1,L}^*}^{L-1}=\mathcal{E}_{\delta+\tau^*_{1,L}}[Y^{(L-1)}_{\tau^*_{2,L}}]=\mathcal{E}_{\delta+\tau^*_{1,L}}[X^{(L-1)}_{\tau^*_{2,L}}]=\mathcal{E}_{\delta+\tau^*_{1,L}}[X_{\tau^*_{2,L}}+\mathcal{E}_{\tau^*_{2,L}}[Y^{(L-2)}_{\delta+\tau^*_{2,L}}]].
   	\end{equation}
   	Combining Equation \eqref{a4} and \eqref{a5}, we get that 
   	\begin{displaymath}
   		\mathcal{E}[\sum_{j=1}^L X_{\tau_j}]\leq Y_0^{(L)}\leq \mathcal{E}[X_{\tau^*_{1,L}}+X_{\tau^*_{2,L}}+Y^{(L-2)}_{\delta+\tau^*_{2,L}}].
   	\end{displaymath}
   	Repeating the above analysis, we finally obtain that 
   	\begin{displaymath}
   		\mathcal{E}[\sum_{j=1}^L X_{\tau_j}]\leq Y_0^{(L)}\leq\mathcal{E}[\sum_{j=1}^L X_{\tau^*_{j,L}}].
   	\end{displaymath}
   	Since $(\tau_1,\cdots,\tau_L)\in \bar{\mathcal{S}}_0(L,\delta)$ is chosen arbitrarily, we obtain the desired result.
   \end{proof}

   We then give an alternative approach for the proof of Theorem \ref{t} based on the result obtained in the discrete-time case. Note that the stopping times defined in \eqref{e31} satisfy the recursive condition, that is, for any $j=2,\cdots,L$ 
   \begin{displaymath}
   	\tau_{1,j}^*(t)=\sigma_j^*(t),\  \tau_{d,j}^*(t)=\sigma^*_{j-d+1}(\tau^*_{d-1,j}(t)+\delta), \ d=2,\cdots,j.
   \end{displaymath}
   Similar with the proof of Lemma \ref{l3}, we have $\tau_{d,j}^*(t)=\tau_{d-1,j-1}^*(\sigma^*_j(t)+\delta)$ for any $d=2,\cdots,j$ and for any $j=2,\cdots,L$, $t\in[0,T]$
   \begin{displaymath}
   	Y^{(j)}_t=\mathcal{E}_t[\sum_{i=1}^j X_{\tau^*_{i,j}(t)}].
   \end{displaymath}
    Applying this result, Theorem \ref{t} can be derived easily from the fact that $Z^L_t$ coincides with $Y^{(L)}_t$ for any $t\in[0,T]$.
    
    \begin{theorem}
    	Let Assumption \ref{a1} hold. Then, for any $t\in[0,T]$ and $j=1,\cdots,L$, we have
    	\begin{displaymath}
    	Z^j_t=Y^{(j)}_t.
    	\end{displaymath}
    \end{theorem}
    
    \begin{proof}
    	Noting that $(\tau^*_{1,j}(t),\cdots,\tau^*_{j,j}(t))\in \bar{\mathcal{S}}_t(j,\delta)$, it follows that 
    	\begin{displaymath}
    	Y^{(j)}_t=\mathcal{E}_t[\sum_{i=1}^j X_{\tau^*_{i,j}(t)}]\leq\esssup_{\tau^{(j)}(t)\in\bar{\mathcal{S}}_t(j,\delta)}\mathcal{E}_t[\sum_{i=1}^j X_{\tau_i(t)}]= Z^j_t.
    	\end{displaymath}
    	
    	Now we are in a position to show the inverse inequality. It is obvious to check that $Z^1_t\leq Y^{(1)}_t$, for any $t\in[0,T]$ (more precisely, they are equal). Assume that for $j=n$, we have $Z^n_t\leq Y^{(n)}_t$, for any $t\in[0,T]$. Then we obtain that 
    	\begin{align*}
    	Z^{n+1}_t&=\esssup_{\tau^{(n+1)}(t)\in\bar{\mathcal{S}}_t(n+1,\delta)}\mathcal{E}_t[\sum_{l=1}^{n+1}X_{\tau_l(t)}]\\
    	&=\esssup_{\tau^{(n+1)}(t)\in\bar{\mathcal{S}}_t(n+1,\delta)}\mathcal{E}_t[X_{\tau_1(t)}+\mathcal{E}_{\tau_1(t)+\delta}[\sum_{l=2}^{n+1}X_{\tau_l(t)}]]\\
    	&\leq \esssup_{\tau^{(n+1)}(t)\in\bar{\mathcal{S}}_t(n+1,\delta)}\mathcal{E}_t[X_{\tau_1(t)}+\esssup_{\widetilde{\tau}^{(n)}(\tau_1(t)+\delta)\in\bar{\mathcal{S}}_{\tau_1(t)+\delta}(n,\delta)}\mathcal{E}_{\tau_1(t)+\delta}[\sum_{l=1}^n X_{\widetilde{\tau}_l(\tau_1(t)+\delta)}]]\\
    	&=\esssup_{\tau_1(t)\in \bar{\mathcal{S}}_{t}}\mathcal{E}_t[X(\tau_1(t))+Z^n_{\tau_1(t)+\delta}]
    	\leq \esssup_{\tau_{1}(t)\in \bar{\mathcal{S}}_{t}}\mathcal{E}_t[X_{\tau_1(t)}+Y^{(n)}_{\tau_1(t)+\delta}]=Y^{(n+1)}_t.
    	\end{align*} 
    	The proof is complete.
    \end{proof}

   In the following, we will show that the value function of the optimal multiple stopping problem in the discrete-time case converges to the one of the continuous time case. For this purpose, set $\mathcal{I}_n=\{0,\frac{1}{2^n}\wedge T,\frac{2}{2^n}\wedge T,\cdots,T\}$ and for any stopping time $\tau$, 
   \begin{equation}\label{approximate}
     \tau^n=\sum_{k=1}^{M}\frac{k}{2^n}I_{\{\frac{k-1}{2^n}\leq \tau<\frac{k}{2^n}\}}+TI_{\{\tau\geq \frac{M}{2^n}\}},
   \end{equation}
   where $M$ satisfies $\frac{M}{2^n}\leq T<\frac{M+1}{2^n}$. Let $\mathcal{S}^n_t(L,\delta)$ be the set of stopping times $\tau^{(L),n}=(\tau^n_1(t),\cdots,\tau^n_L(t))$ such that $\tau^n_j(t)$ takes values in $\mathcal{I}_n$ for all $1\leq j\leq L$, $t\leq \tau^n_1(t)$, for any $2\leq j\leq L$, $\tau^n_{j-1}(t)+\delta\leq \tau^n_j(t)$. The value function of the multiple stopping problem can be defined as the following:
   \begin{displaymath}
   	  Z^n_L(t)=\esssup_{\tau^{(l),n}(t)\in\mathcal{S}^n_t(L,\delta)}\mathcal{E}_t[\sum_{j=1}^{L}X_{\tau^n_j(t)}].
   \end{displaymath}
   \begin{proposition}
   	  Under Assumption \ref{a1}, we have 
   	  \begin{displaymath}
   	  	Z^L_t=\lim_{n\rightarrow\infty} Z^n_l(t).
   	  \end{displaymath}
   \end{proposition}
   
   \begin{proof}
   	 It is easy to check that for any $n\in\mathbb{N}$, $\mathcal{S}^n_t(L,\delta)\subset \mathcal{S}^{n+1}_t(L,\delta)\subset \bar{\mathcal{S}}_t(L,\delta)$. Therefore, we have $Z^n_L(t)\leq Z^{n+1}_L(t)\leq Z^L_t$, which implies that $Z_t^L\geq \lim_{n\rightarrow\infty} Z^n_L(t)$. On the other hand, for any $(\tau_1,\cdots,\tau_L)\in \bar{\mathcal{S}}_t(L,\delta)$, we may approximate $\tau_j$ from the right by $\tau^n_j$ as Equation \eqref{approximate}, for all $1\leq j\leq L$. For $n$ large enough such that $\frac{1}{2^{n+1}}<\delta$, we can check that $(\tau^n_1,\cdots,\tau^n_L)\in \mathcal{S}_t^n(L,\delta)$. By the continuity of $X$ and dominated convergence theorem, we get that 
   	 \begin{displaymath}
   	 	\mathcal{E}_t[\sum_{j=1}^{L}X_{\tau_j}]=\lim_{n\rightarrow\infty}\mathcal{E}_t[\sum_{j=1}^{L}X_{\tau^n_j}]\leq \lim_{n\rightarrow\infty} Z^n_L(t).
   	 \end{displaymath}
   	 Since $(\tau_1,\cdots,\tau_L)$ is arbitrarily chosen, the above equation yields that $Z_t^L\leq \lim_{n\rightarrow\infty} Z^n_L(t)$. The proof is complete.
   \end{proof}

   \begin{remark}
   	 In fact, Theorem 3.1 in \cite{CR} still holds true without the assumption that $g$ is concave in $z$, which implies that all the results in this section hold under this weak assumption.
   \end{remark}

   \section{Some properties of the value function and optimal stopping times}

   We first consider a special case that $\delta=0$, which means that the holder of the contract can exercise more than one rights at the same time. This is trivial for the linear expectation case since the optimal multiple stopping times should satisfy $\tau_{i,L}^*=\tau^*$, $i=1,\cdots,L$, where $\tau^*$ is the optimal stopping time for the single stopping problem. However, due to the nonlinearity of $g$-expectation, this property is not obvious for the Knightian uncertainty case. We can only prove this property for some typical situations.
   
   \begin{proposition}
   	 Let $g$ satisfy Assumption \ref{a7} (without concavitiy) and we assume additionally, $g$ is positive homogeneous and sub-additive. Suppose that the nonngative reward process $X$ is RCLL and left-continous in $g$-expectation with $E[|\bar{X}|^2]<\infty$, where $\bar{X}=\sup_{t\in[0,T]}X_t$. Consider the optimal multiple stopping problem
   	 \begin{displaymath}
   	 	Z^L_t=\esssup_{\tau^{(L)}(t)\in\bar{\mathcal{S}}_t(L)}\mathcal{E}_t[\sum_{i=1}^L X_{\tau_i(t)}],
   	 \end{displaymath}
   	 where $\bar{\mathcal{S}}_t(L)=\bar{\mathcal{S}}_t(L,0)$. Then, the optimal multiple stopping times $(\tau^*_1(t),\cdots,\tau^*_L(t))$ satisfy $\tau^*_i(t)=\tau^*_t$, $i=1,\cdots,L$, where $\tau^*_t$ is optimal for 
   	 \begin{displaymath}
   	 	V_t=\esssup_{\tau\in\bar{\mathcal{S}}_t}\mathcal{E}_t[X_\tau].
   	 \end{displaymath}
   \end{proposition}

   \begin{proof}
   	 It is easy to verify that 
   	 \begin{align*}
   	 	\mathcal{E}_t[L X_{\tau^*_t}]\leq Z_t^L\leq \esssup_{\tau^{(L)}(t)\in\bar{\mathcal{S}}_t(L)}\sum_{i=1}^L \mathcal{E}_t[X_{\tau_i(t)}]\leq L\esssup_{\tau\in\bar{\mathcal{S}}_t}\mathcal{E}_t[X_\tau]=L V_t=L\mathcal{E}_t[X_{\tau^*_t}]=\mathcal{E}_t[L X_{\tau^*_t}].
   	 \end{align*}
   	 Hence, the result follows.
   	 \end{proof}

   \begin{proposition}\label{p1}
   	If $g$ is a sublinear function satisfying (1), (2) in Assumption \ref{a7} together with $g(t,0)=0$, then for any $i=1,2,\cdots,L$, we have
   	\begin{displaymath}
   	Y_t^{(i)}\leq Y_t^{(1)}+\mathcal{E}_t[Y_{t+\delta}^{(i-1)}].
   	\end{displaymath}
   	Consequently, we have $\tau_{i,L}^*(t)\leq \zeta_i^*(t)$, for any $i=1,\cdots,L$ and $t\in[0,T]$, where
   	\begin{displaymath}
   	\zeta_i^*(t)=\inf\{t\geq \delta+\tau_{i-1,L}^*(t): Y_t^{(1)}=X_t^{(1)}\}I_{\{\delta+\tau_{i-1,L}^*(t)\leq T\}}+(\delta+\tau_{i-1,L}^*(t))I_{\{\delta+\tau_{i-1,L}^*(t)>T\}}.
   	\end{displaymath}
   \end{proposition}
   
   \begin{proof}
   	It is easy to verify that  $\mathcal{E}$ is sublinear under the new assumption on $g$. By Theorem 3.1 in \cite{CR}, we know that for any $i=0,1,\cdots,L$, $Y^{(i)}$ is the smallest $g$-supermartingale dominating the process $X^{(i)}$. Let $V_t^{(i)}=Y_t^{(1)}+\mathcal{E}_t[Y_{t+\delta}^{(i-1)}]$. It is easy to check that $V^{(i)}\geq X^{(i)}$. We claim that $V^{(i)}$ is a $g$-supremartingale. Indeed, and for any $0\leq s\leq t\leq T-\delta$, we have
   	\begin{displaymath}
   	\mathcal{E}_s[V_t^{(i)}]\leq \mathcal{E}_s[Y_t^{(1)}]+\mathcal{E}_s[\mathcal{E}_t[Y^{(i-1)}_{t+\delta}]]=\mathcal{E}_s[Y_t^{(1)}]+\mathcal{E}_{s}[\mathcal{E}_{s+\delta}[Y^{(i-1)}_{t+\delta}]]\leq Y_s^{(1)}+\mathcal{E}_s[Y_{s+\delta}^{(i-1)}]=V_s^{(i)}.
   	\end{displaymath}
   	For any  $T-\delta<t\leq T$ and $s\leq t$, we have
   	\begin{displaymath}
   	\mathcal{E}_s[V_t^{(i)}]=\mathcal{E}_s[Y_t^{(1)}]\leq Y_s^{(1)}\leq V_s^{(i)}.
   	\end{displaymath}
   	Therefore, $V^{(i)}$ is a $g$-supermartingale dominating the process $X^{(i)}$, which implies that $Y^{(i)}\leq V^{(i)}$.
   	
   	On the set that $\{\delta+\tau_{i-1,L}^*(t)>T\}$, it is obvious that $\tau_{i,L}^*(t)= \zeta_i^*(t)$. On the set $\{\delta+\tau_{i-1}^*(t)\leq T\}$, it follows that $Y^{(1)}_{\zeta_i^*}=X^{(1)}_{\zeta_i^*}=X_{\zeta_i^*}$. Consequently, we have
   	\begin{displaymath}
   	X^{(L-i+1)}_{\zeta_i^*}=X_{\zeta_i^*}+\mathcal{E}_{\zeta_i^*}[Y^{(L-i)}_{\zeta_i^*+\delta}]=Y^{(1)}_{\zeta_i^*}+\mathcal{E}_{\zeta_i^*}[Y^{(L-i)}_{\zeta_i^*+\delta}]\geq Y^{(L-i+1)}_{\zeta_i^*}\geq X^{(L-i+1)}_{\zeta_i^*}.
   	\end{displaymath}
   	By the definition of $\tau_{i,L}^*(t)$, we finally get the desired result.
   \end{proof}
   
   \begin{remark}
   	In fact, $\tau_{i,L}^*(t)$ is the optimal exercise time for the $i$-th right for the optimal multiple stopping problem, while $\zeta_i^*(t)$ can be regarded as the optimal stopping time for the optimal single problem starting from time $\delta+\tau_{i-1,L}^*(t)$, which is the first time that one could exercise her $i$-th right for the optimal multiple stopping problem. The above proposition indicates that if the agent has more than one exercise right, the optimal exercise time is earlier than the one for the single stopping problem.
   \end{remark}

     If the reward process $X$ is a $g$-submartingale in strong sense (we do not need to assume that $X$ is left-continuous and the terminal value is $0$), we have $\mathcal{E}_\sigma[X_\tau]\geq X_\sigma$ for any stopping times $\sigma,\tau$ with $\sigma\leq \tau$, which implies that $\mathcal{E}_t[X_T]\geq \mathcal{E}_t[X_\tau]$ for any $\tau\in\mathcal{S}_t$. Therefore, we have 
   \begin{displaymath}
   Z_t^1=\esssup_{\tau\in\bar{\mathcal{S}}_t}\mathcal{E}_t[X_\tau]=\mathcal{E}_t[X_T],
   \end{displaymath}
   that is, the terminal time $T$ is optimal. Furthermore, assume that the $g$-expectation is super-additive. For the multiple stopping case, suppose that $d=2,\cdots,L$ satisfies $\delta(d-1)\leq T-t<\delta d$. We claim that 
   \begin{displaymath}
   Z_t^L=\esssup_{\tau^{(L)}(t)\in\bar{\mathcal{S}}_t(L,\delta)}\mathcal{E}_t[\sum_{i=1}^L X_{\tau_i(t)}]=\mathcal{E}_t[\sum_{i=0}^{d-1} X_{T-i\delta}].
   \end{displaymath}
   First, under the constraints that $\delta(d-1)\leq T-t<\delta d$, we have $\tau_{d+1}(t)>T$, which implies that $X_{\tau_j(t)}=0$ for any $j=d+1,\cdots,l$. It follows that 
   \begin{displaymath}
   Z_t^L=\esssup_{\tau^{(d)}(t)\in\bar{\mathcal{S}}_t(d,\delta)}\mathcal{E}_t[\sum_{i=1}^d X_{\tau_i(t)}]
   \end{displaymath}
   It remains to prove that for any $\tau^{(d)}(t)\in\mathcal{S}_t(d,\delta)$, we have $\mathcal{E}_t[\sum_{i=1}^d X_{\tau_i(t)}]\leq \mathcal{E}_t[\sum_{i=0}^{d-1} X_{T-i\delta}]$. For any $\xi\in L^2(P)$, it is easy to check that the process $\widetilde{X}$ is a $g$-submartingale, where $\widetilde{X}_t=X_t+\mathcal{E}_t[\xi]$. By the submartingale property, we may obtain  that 
   \begin{align*}
   \mathcal{E}_t[\sum_{i=1}^d X_{\tau_i(t)}]&=\mathcal{E}_t[\sum_{i=1}^{d-1} X_{\tau_i(t)}+\mathcal{E}_{\tau_{d-1}(t)}[X_{\tau_d(t)}]]\leq \mathcal{E}_t[\sum_{i=1}^{d-1} X_{\tau_i(t)}+\mathcal{E}_{\tau_{d-1}(t)}[X_{T}]]\\
   &=\mathcal{E}_t[\sum_{i=1}^{d-2} X_{\tau_i(t)}+\mathcal{E}_{\tau_{d-2}(t)}[X_{\tau_{d-1}(t)}+\mathcal{E}_{\tau_{d-1}(t)}[X_T]]]\\
   &\leq \mathcal{E}_t[\sum_{i=1}^{d-2} X_{\tau_i(t)}+\mathcal{E}_{\tau_{d-2}(t)}[X_{T-\delta}+\mathcal{E}_{T-\delta}[X_T]]]\\
   &=\mathcal{E}_t[\sum_{i=1}^{d-3} X_{\tau_i(t)}+\mathcal{E}_{\tau_{d-3}(t)}[X_{\tau_{d-2}(t)}+\mathcal{E}_{\tau_{d-2}(t)}[X_{T-\delta}+X_T]]]\\
   &\leq \cdots\leq \mathcal{E}_t[\sum_{i=0}^{d-1} X_{T-i\delta}],
   \end{align*}
   which generalizes the following result.
   
   \begin{proposition}
   	Suppose that the reward process $X$ is a $g$-submartingale with $E[\sup_{t\in[0,T]}|X_t|^2]<\infty$ and $g$ is a superlinear function satisfying (1), (2) in Assumption \ref{a7} together with $g(t,0)=0$. Then, if $d=2,\cdots,L$ satisfies $\delta(d-1)\leq T-t<\delta d$, we have 
   	\begin{displaymath}
   	Z_t^L=\esssup_{\tau^{(L)}(t)\in\bar{\mathcal{S}}_t(L,\delta)}\mathcal{E}_t[\sum_{i=1}^L X_{\tau_i(t)}]=\mathcal{E}_t[\sum_{i=0}^{d-1} X_{T-i\delta}].
   	\end{displaymath}
   \end{proposition}

   Suppose that the $g$-expectation degenerates into the $\kappa$-ignorance $\mathcal{E}^{-\kappa}[\cdot]$ as stated in \cite{CE}, that is, the generator $g$ takes the form that $g(t,z)=-\kappa|z|$. Assume that the payoff process $X$ is given by $\{f(t,S^x_t)\}_{t\in[0,T]}$, where $f(t,x)$ is a continuous, monotonic function with polynomial growth and $f(T,x)=0$, and $S^x$ evolves as the following stochastic differential equation
   \begin{displaymath}
   	 dS^x_t=bS^x_tdt+\sigma S^x_tdB_t, \ S^x_0=x,
   \end{displaymath}
   for some positive constant $b$, $\sigma$. Without loss of generality, we assume that $f$ is increasing in $x$. Let $v^{(1)}:[0,T]\times\mathbb{R}\rightarrow\mathbb{R}$ solve the following Hamilton-Jacobi-Bellman equation
   \begin{equation}\label{eq1}
   	\max_{(t,x)\in[0,T]\times\mathbb{R}}\{f(t,x)-v(t,x), \partial_t v(t,x)+\mathcal{L}v(t,x)-\kappa\sigma(x)\partial_x v(t,x)\}=0,
   \end{equation}  
   with terminal condition $v(T,x)=0$, where $\mathcal{L}=\frac{1}{2}\sigma^2\frac{\partial^2}{\partial_xx}+b\frac{\partial}{\partial_x}$. If we furthermore assume that $v^{(1)}$ and its generalized first derivative satisfy the polynomial growth condition, by Theorem 4.1 in \cite{CR}, we have 
   \begin{equation}\label{eq2}
   	Y^{(1)}_t=\esssup_{\tau\in\mathcal{S}_t}\mathcal{E}^{-\kappa}_t[X_\tau]=v^{(1)}(t,S^x_t)=\esssup_{\tau\in\mathcal{S}_t}{E}^{-\kappa}_t[X_\tau],
   \end{equation} 
   where $E^{-\kappa}_t[\cdot]$ is the conditional expectation taken under $P^{-\kappa}$ whose Girsanov kernel is $-\kappa$, i.e., 
   \begin{displaymath}
   	\frac{dP^{-\kappa}}{dP_0}|_{\mathcal{F}_T}=\exp(-\kappa B_T-\frac{1}{2}\kappa^2 T).
   \end{displaymath}

   Note that the monotonicity of $f$ ensures the monotonicty of $v^{(1)}$. By Theorem 2 in \cite{CKW}, for any fixed $t\in[0,T]$, we have $z_s\geq 0$ a.e. $s\in[0,t+\delta]$, where $z$ is the second component of solution to the following BSDE
   \begin{displaymath}
   	y_s=v^{(1)}(t+\delta, S^x_{t+\delta})-\int_s^{t+\delta}\kappa|z_r|dr-\int_s^{t+\delta}z_rdB_r.
   \end{displaymath}
   By the Girsanov transformation and Markov property, we obtain that
   \begin{align*}
   	y_t&=\mathcal{E}^{-\kappa}_t[v^{(1)}(t+\delta,S^x_{t+\delta})]=E^{-\kappa}_t[v^{(1)}(t+\delta,S^x_{t+\delta})]\\
   	&=E_t[v^{(1)}(t+\delta,\widetilde{S}^x_{t+\delta})]=E[v^{(1)}(t+\delta,\widetilde{S}_\delta^y)]|_{y=\widetilde{S}_t^x}\\
   	&=E^{-\kappa}[v^{(1)}(t+\delta, S_\delta^y)]|_{y=\widetilde{S}^x_t},
   \end{align*}
   where $\widetilde{S}^x$ is the solution of the following SDE:
   \begin{displaymath}
   	 d\widetilde{S}_t^x=(b-\kappa\sigma)\widetilde{S}_t^xdt+\sigma\widetilde{S}_t^x dB_t,\  \widetilde{S}_0^x=x.
   \end{displaymath}
   It is easy to check that $\widetilde{S}^x_t=\exp(-\kappa\sigma t)S_t^x$. Hence, $y_t=E^{-\kappa}[v^{(1)}(t+\delta, S_\delta^{\exp(-\kappa\sigma t)y})]|_{y={S}^x_t}$ Set $f^{(2)}(t,x)=f(t,x)+E^{-\kappa}[v^{(1)}(t+\delta,S_\delta^{\exp(-\kappa\sigma t)x})]$. It is easy to check that $f^{(2)}$ is increasing in $x$ and $X^{(2)}_t=f^{(2)}(t,S^x_t)$. Now let $v^{(2)}$ be the solution of the HJB equation \eqref{eq1} with $f(t,x)=f^{(2)}(t,x)$. Again by Theorem 4.1 in \cite{CR}, if $v^{(2)}$ and its generalized first derivative satisfy the polynomial growth condition, we have
   \begin{equation}\label{eq3}\begin{split}
   	Y^{(2)}_t&=\esssup_{\tau\in\mathcal{S}_t}\mathcal{E}^{-\kappa}_t[X^{(2)}_\tau]=v^{(2)}(t,S^x_t)=\esssup_{\tau\in\mathcal{S}_t}E^{-\kappa}_t[X^{(2)}_\tau]\\
   	&=\esssup_{\tau\in\mathcal{S}_t} E^{-\kappa}_t[X_\tau+E^{-\kappa}_{\tau}[Y^{(1)}_{\tau+\delta}]].
   \end{split}\end{equation}
   Recalling Equation \eqref{eq2}, we conclude that $Y^{(1)}$ is a supermartingale dominating the process $X=X^{(1)}$ under $P^{-\kappa}$. Hence, for any stopping times $\tau,\sigma$ with $\tau+\delta\leq \sigma$, we have
   \begin{equation}\label{eq4}
   	E^{-\kappa}_\tau[Y^{(1)}_{\tau+\delta}]\geq E^{-\kappa}_\tau[E^{-\kappa}_{\tau+\delta}[Y^{(1)}_\sigma]]\geq {E}^{-\kappa}_\tau[X^{(1)}_\sigma]={E}^{-\kappa}_\tau[X_\sigma].
   \end{equation}
   Combining Equation \eqref{eq3} and \eqref{eq4}, we derive that 
   \begin{displaymath}
   \esssup_{\tau^{(2)}(t)\in\mathcal{S}_t(2,\delta)}\mathcal{E}^{-\kappa}_t[X_{\tau_1}+X_{\tau_2}]=	Y_t^{(2)}\geq \esssup_{\tau\in\mathcal{S}_t,\sigma\geq \tau+\delta}E^{-\kappa}_t[X_\tau+X_\sigma]=\esssup_{\tau^{(2)}(t)\in\mathcal{S}_t(2,\delta)}{E}^{-\kappa}_t[X_{\tau_1}+X_{\tau_2}],
   \end{displaymath}
    which implies that 
    \begin{displaymath}
    	Y_t^{(2)}=\esssup_{\tau^{(2)}(t)\in\mathcal{S}_t(2,\delta)}{E}^{-\kappa}_t[X_{\tau_1}+X_{\tau_2}].
    \end{displaymath}
    
    \begin{theorem}
    	Assume that the payoff function $f$ is continuous, increasing in $x$ and satisfies the polynomial growth condition. Denote by $v^{(i)}(t,x)$ the solution of HJB equation \eqref{eq1} with $f(t,x)$ is replaced by $f^{(i)}(t,x)$, where $f^{(1)}(t,x)=f(t,x)$ and
    	\begin{displaymath}
    		f^{(i)}(t,x)=f(t,x)+E^{-\kappa}[v^{(i-1)}(t+\delta,S_\delta^{\exp(-\kappa\sigma t)x})], \ i=2,\cdots,L.
    	\end{displaymath}
    	Assume that $v^{(i)}$ and its generalized first derivative satisfy the polynomial growth condition. Then the optimal multiple stopping problem under $\kappa$-ignorance coincides with the one under $P^{-\kappa}$ with value function $Z_t^i=Y_t^{(i)}=v^{(i)}(t,S_t^x)$.
    \end{theorem}

\end{document}